\documentclass[11pt,a4paper,oneside]{article}
\usepackage{setspace}
\usepackage{fancyhdr}
\usepackage{tabularx}
\usepackage[a4paper,left=1cm,right=1cm,top=1cm,bottom=2cm]{geometry}

\usepackage[T2A]{fontenc}
\usepackage[utf8]{inputenc}
\usepackage{graphicx}
\usepackage[affil-it]{authblk}
\usepackage{amssymb}
\usepackage{amsmath}
\usepackage{amsthm}
\usepackage{subcaption}
\usepackage[natbib=true,style=numeric,sorting=none,citestyle=numeric-comp,style=ieee,backend=biber]{biblatex}
\usepackage{comment}

\usepackage{mathrsfs}
\usepackage{epigraph}

\newtheorem{theorem}{Theorem}
\newtheorem{proposition}{Proposition}
\newtheorem{lemma}{Lemma}

\title{\textbf{Asymptotically Stable Non-Falling Solutions\\ of the Kapitza-Whitney Pendulum}}

\author{\textbf{Ivan Polekhin}%
  \thanks{Electronic address: \texttt{ivanpolekhin@mi-ras.ru}}}
\affil{Steklov Mathematical Institute of Russian Academy of Sciences,\\ Moscow, Russia\\
Moscow Institute of Physics and Technology,\\
Dolgoprudny, Russia}

\date{\textbf{24.04.2021 -- 16.05.2021}}

\date{}

\begin{document}

\maketitle

\section*{Abstract}

The planar inverted pendulum with a vibrating pivot point in the presence of an additional horizontal force field is studied. The horizontal force is not assumed to be small or rapidly oscillating. We assume that the pivot point of the pendulum rapidly oscillates in the vertical direction and the period of these oscillations is commensurable with the period of horizontal force. This system can be considered as a strong generalization of the Kapitza pendulum. Previously it was shown that for any horizontal force there always exists a non-falling periodic solution in the considered system. In particular, when there is no horizontal force, this periodic solution is the vertical upward position. In the paper we present analytical and numerical results concerning the existence of asymptotically stable non-falling periodic solutions in the system.

\vspace{10pt}
\noindent\textbf{Keywords: } forced oscillation, the Kapitza pendulum, the Whitney pendulum, stabilization, vibration.

\section{Introduction}

The problem of motion of a pendulum with a vibrating pivot point is one of the well-studied non-linear dynamical systems. Like other classical systems --- as examples we can mention here the Duffing equation or the Van der Pol oscillator --- this system, firstly, is quite simple, which allows one to study this system analytically, and secondly, some non-trivial dynamical effects can be observed within the framework of this system.

Similarly to the the Duffing equation or the Van der Pol oscillator, the study of the pendulum with a vibrating pivot point goes beyond the boundaries of initial statements of the problems. Historically the problem of motion of a pendulum with a vibrating base goes back to A. Stephenson \cite{stephenson1908xx}, N.N. Bogolyubov \cite{bogolyubov1950perturbation} and P.L. Kapitza \cite{kapitsa1951dynamic,kapitsa1951pendulum}. Let us also mention some recent papers on the considering problem \cite{araujo2021parametric,artstein2021pendulum,belyaev2021classical,cabral2021parametric}.
The comprehensive bibliography can hardly be listed here. However, one can find a relatively full overview of the papers on the topic in \cite{butikov2001dynamic,samoilenko1994nn}, including some references on the history of the problem.

In the presence of friction, for the pendulum with a vibrating base, one can show that the vertical upward position (usually unstable) becomes asymptotically stable, provided that the parameters of the system satisfy some conditions. The proof follows from the classical method of averaging. If there is no friction in the system, the proof is more complicated and also involves KAM theory \cite{bardin1995stability}. As a natural generalization of this system, one can consider a planar mathematical pendulum with a vibrating point of suspension and in a periodic in time horizontal force field.

If the horizontal force is rapidly oscillating, then we again can study this system by means of the averaging theory. This problem has been considered previously (see, for instance, \cite{burd2007method}). In a more general statement of the problem, it is not assumed that the horizontal force has a small period. In \cite{polekhin2020method} it was shown that for any horizontal periodic force, such that its period is commensurate to the period of the vertical oscillation, there always exists a periodic non-falling solution. Here we say that the considered solution is non-falling if the rod of the pendulum never becomes horizontal and remains in the upper half-plane. If the pivot point does not oscillate, we obtain the so-called Whitney pendulum (for details, see \cite{polekhin2020method}). When there is no horizontal force, then we can consider the vertical upward equilibrium as a non-falling solution. Is it always possible to make this non-falling solution stable by changing the parameters of the vertical vibration? We will study this question in the paper.

In the first section we present the equations of motion and some analytical results on the problem. In particular, we will consider the case when the horizontal force is weak. The case when the horizontal force is not weak is considered numerically.

 To be more precise, we consider the following two problems:
\begin{enumerate}
    \item Given any $2\pi$-periodic non-falling solution one can uniquely determine the corresponding horizontal force, which guarantees the existence of this solution. For which values of the amplitude of the vertical oscillation does this solution become stable?
    \item Given any $2\pi$-periodic horizontal force. Again, we find numerically the values of the amplitude such that there exists a periodic stable non-falling solution.
\end{enumerate}

In the conclusion section we will briefly discuss some conjectures related to the problem. 

\section{Equations of motion}

\subsection{Planar pendulum}

Let us consider a point of mass $m$ moving along a circle of radius $l$ in the presence of a gravitational field. We assume that there is a force of viscous friction acting on the point. If the circle is placed in a vertical plane, then this system is the classical planar mathematical pendulum (Fig. 1). If the pivot point of the pendulum moves along the vertical axis and the law of motion is given by a function $h(t)$, then the position of the point is defined as follows
\begin{align*}
    &x = l \sin \varphi,\\
    &y = h(t) - l \cos \varphi.
\end{align*}
Here $\varphi$ is the angle between the vertical axis and the rod of the pendulum. The kinetic and the potential energy ($T$ and $\Pi$, respectively) of the system have the usual form
\begin{align*}
    T = \frac{m}{2}(ml^2 \dot\varphi^2 + 2\dot h \dot \varphi l \sin\varphi), \quad \Pi = - mgl \cos\varphi.
\end{align*}
Here $g$ is the acceleration of gravity. The usual Lagrangian equations of motion for this system (with the Lagrangian function $L = T - \Pi$) have the following form
\begin{align*}
    ml^2 \ddot \varphi + m \ddot h l \sin \varphi = -mgl \sin \varphi.
\end{align*}

\begin{figure}[h!]
  \centering
  \includegraphics[width=0.66\linewidth]{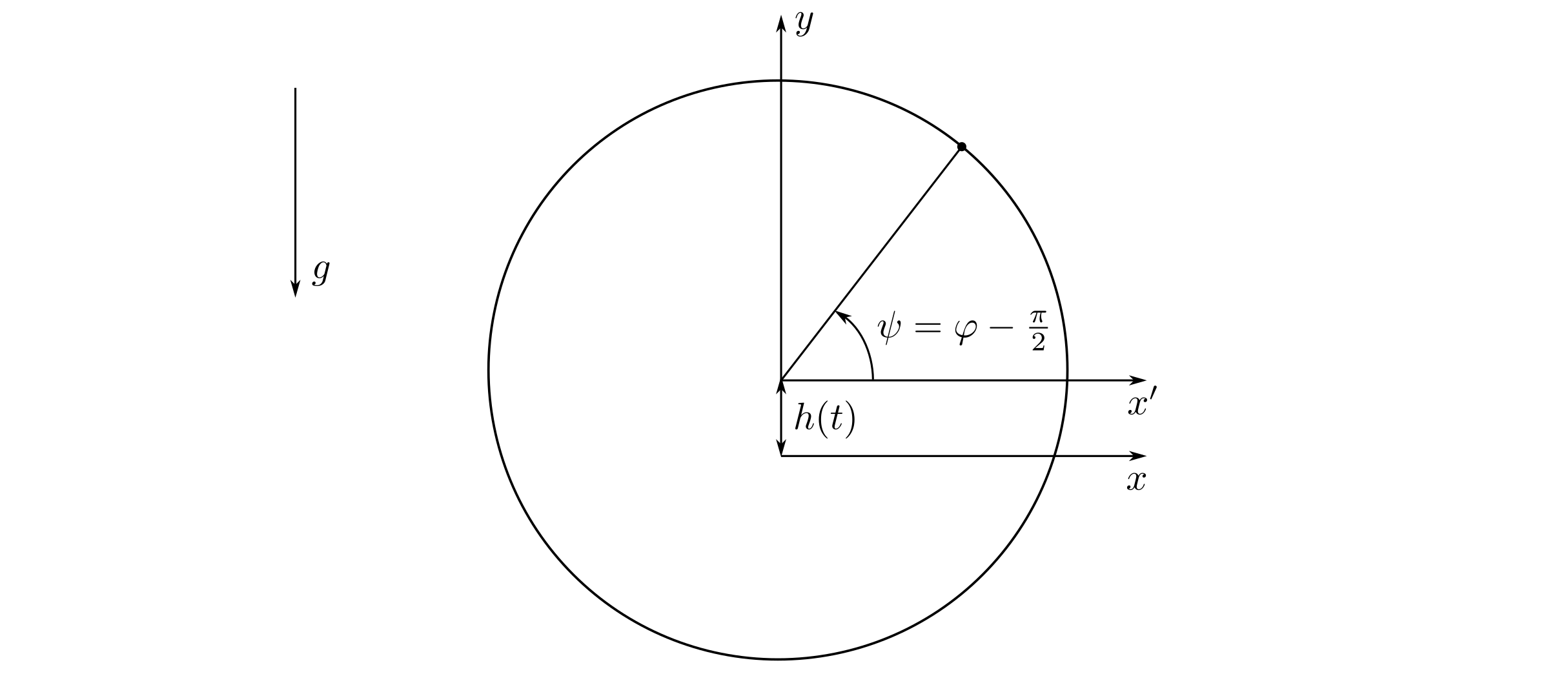}
  \caption{Inverted pendulum with a vibrating pivot point.}
  \label{fig:sfig2}
\end{figure}

If we additionally consider viscous friction and a horizontal force acting on the system, then the above equation have to be rewritten as follows
\begin{align}
    \label{eq1}
    ml^2 \ddot \varphi + m \ddot h(t) l \sin \varphi = -mgl \sin \varphi - \mu l^2 \dot \varphi + F(t) l \cos\varphi.
\end{align}
Here $\mu$ is the coefficient of viscous friction and $F(t)$ is a function of time. Without any loss of the generality, we can assume that $m = g  = l = 1$. Then system \eqref{eq1} can be rewritten as follows
\begin{align}
\begin{split}
    \label{eq2}
    &\dot \varphi = p - \dot h \sin \varphi,\\
    &\dot p = -\sin \varphi - \mu p + \mu \dot h \sin \varphi + \dot h p \cos \varphi - \dot h^2 \sin \varphi \cos\varphi + F \cos\varphi.
\end{split}
\end{align}
Below we assume that $h(t)$ is a rapidly oscillating function. To be more precise, we assume that this function has the form 
$$
h(t) = a \varepsilon \sin \frac{t}{\varepsilon}.
$$
For small $\varepsilon > 0$ it is possible to consider the averaged system
\begin{align}
\begin{split}
    \label{eq3}
    &\dot \varphi = p,\\
    &\dot p = -\sin \varphi - \mu p - (a^2/4)  \sin 2\varphi + F(s) \cos\varphi,\\
    &\dot s = 1.
\end{split}
\end{align}

Here $s$ is an artificial time-like parameter (slow time) which is introduced in order to distinguish two types of time-dependent functions.

\section{Analytical results}

\subsection{Periodic solutions of the averaged system}

In this section we obtain some analytical results on the existence of a periodic solution (and its stability) for system \eqref{eq2}. First, we will consider the same question for averaged system \eqref{eq3}. For instance, by means of the classical Poincare method we will prove the existence of a periodic asymptotically stable solution, provided that the perturbation is small and non-autonomous. Therefore, from a result by N.N. Bogolyubov on the existence of an integral manifold, we obtain the existence of a periodic asymptotically stable solution in the original system.

We will also consider the case when the perturbation in \eqref{eq3} is not small. Here we will use a result by P. Torres on the existence and stability of periodic solutions for the Duffing equation \cite{torres2004existence}. 

Everywhere below we assume that all functions are real analytic, except for the cases when we explicitly assume that some function belongs to another class of regularity.

\begin{proposition}
Let function $F(t)$ in \eqref{eq3} be as follows: $F(t) = c f(t)$, where $f$ is a $2\pi$-periodic function. Let $\mu > 0$ and $a^2 > 2$, then for sufficiently small  $c > 0$, for system \eqref{eq3}, there exists a $2\pi$-periodic asymptotically stable solution, and this solution coincides with the vertical upward equilibrium for $c = 0$.
\end{proposition}
\begin{proof}
Let us denote the phase variable by $x = (\varphi, p)$, and the right hand side by $G(t,x,c)$. The standard method of proving the existence of a periodic solution for small $c$ is based on the application of the implicit function theorem to the equation
$$
x(2\pi, x_0; c)- x_0 = 0,
$$
here by $x(t, x_0; c)$ we denote a the solution of the equation
$$
\dot x = G(t,x,c)
$$
with initial condition $x(0) = x_0$. For $c = 0$ we have a trivial solution (the vertical equilibrium of the pendulum). For small $c \ne 0$ there exists a function $x_0(c)$ such that $x(2\pi, x_0(c), c) - x_0(c) = 0$, provided that
$$
\det \left( \frac{\partial x (2\pi, x_0; 0)}{\partial x_0} - E \right) \ne 0.
$$
Here $E$ is the identity matrix. The determinant equals zero when the system in variations (w.r.t. solution $\varphi = \pi$, $p = 0$) has nontrivial periodic solutions:
\begin{align*}
\begin{split}
    &\dot \varphi = p,\\
    &\dot p = -\varphi\left(1 + \frac{a^2}{2}\right) - \mu p.
\end{split}
\end{align*}
It can be easily shown that for function
$$
H = \frac{p^2}{2} + \left(1 + \frac{a^2}{2}\right) \frac{\varphi^2}{2}
$$
we almost everywhere have $\dot H = -\mu p^2 < 0$. Therefore, there exist no non-trivial periodic solutions for the system in variations.

The obtained solutions will be asymptotically stable. Indeed, the vertical position is asymptotically stable for $c=0$. If we consider this equilibrium as a periodic solution, then we conclude that its characteristic multipliers are strictly inside the unit circle. From the continuous dependence of the characteristic multipliers on $c$, we obtain that for small $c$ the considered solution is also asymptotically stable. 
\end{proof}

When the friction is also small, i.e. if we put $c = \mu$  in \eqref{eq3}, then it is also possible to prove the existence of an asymptotically stable $2\pi$-periodic solution for small positive values of $\mu$.  

This result follows from the following theorem proved by Lyapunov \cite{malkin1949methods}.

\begin{theorem}
\label{th1}

Let us have a non-autonomous Lyapunov system on a plane, i.e. a system of differential equations of the following form
\begin{align}
\begin{split}
    \label{eq7}
    &\dot x = -\lambda y + f_x(x,y) + \mu F_x(t,x,y,\mu),\\
    &\dot y = \lambda x + f_y(x,y) + \mu F_y(t,x,y,\mu).
\end{split}
\end{align}
Here $F_x, F_y$ are $2\pi$-periodic functions in $t$. Functions $f_x$ и $f_y$ are of order no less than $two$ in $x, y$. We also assume that
$$
f_x = - \frac{\partial S}{\partial y}, \quad f_y = \frac{\partial S}{\partial x}
$$
If $\lambda \not\in \mathbb{Z}$, then there exists a unique $2\pi$-periodic solution of system \eqref{eq7}. This solution depends analytically on $\mu$ for small $\mu$, and for $\mu = 0$ we obtain the trivial solution $x = y = 0$.
Moreover, if 
$$
\mu \int\limits_0^{2\pi}\left.\left( \frac{\partial f}{\partial x} + \frac{\partial F}{\partial y} \right)\right|_{x=y=\mu=0}\, dt < 0,
$$
then all solutions in the obtained one-parameter family of periodic solutions are asymptotically stable. 
\end{theorem}

If we apply this theorem to system \eqref{eq3}, we obtain the following proposition.
\begin{proposition}
Let $(1 + a^2/2)^{1/2} \notin \mathbb{Z}$, $F(t) = \mu f(t)$, where $f(t)$ is a $2\pi$-periodic function. Then for small $\mu > 0$ there is a $2\pi$-periodic asymptotically stable solution of \eqref{eq3}.
\end{proposition}

\begin{proof}
System \eqref{eq3} can be presented in the following form
\begin{align*}
\begin{split}
    &\dot x = -(1+ a^2/2)^{1/2} y + X(y) - \mu y + \mu f(t)  \cos(y/(1+a^2/2)^{1/2}),\\
    &\dot y = (1+ a^2/2)^{1/2} x.
\end{split}
\end{align*}
Where $X(y)$ is an analytic function of order no less than two in $y$. If $\mu = 0$, then this system is Hamiltonian. From Theorem \ref{th1} we obtain the existence of a periodic solution. Since $\mu > 0$, then 
$$
-2 \pi \mu < 0.
$$
Therefore, this solution is asymptotically stable.
\end{proof}

Let us now consider the case when it is not assumed that $F(t)$ is small. We will use the following version of a result by P. Torres \cite{torres2004existence} (see also \cite{njoku2003stability}).

Given a Duffing equation
\begin{align}
\label{duff}
\ddot x + c \dot x + g(t,x) = 0,
\end{align}
where $c > 0$, $g \colon \mathbb{R}^2 \to \mathbb{R}$ is $2\pi$-periodic in $t$, by $\Omega_{k,c}$ we denote the following set
$$
\Omega_{k,c} = \left\{ f \in L^k(0, 2\pi) \colon f \succ 0, \| f \|_k < \left( 1 + \frac{c^2}{4} \right) K\left( \frac{2k}{k-1} \right) \right\}.
$$

Here $f \succ 0$ means that $f \geqslant 0$ (a.e.) and $f > 0$ in a subset of positive measure.
$$
K(q) = 
\begin{cases}
\frac{1}{q \cdot (2\pi)^{2/q}} \left( \frac{2}{2+q} \right)^{1 - 2/q} \left( \frac{\Gamma(1/2)}{\Gamma(1/2 + 1/q)} \right)^2, & \mbox{ if } 1 \leqslant q < +\infty, \\
2/\pi, & \mbox{ if } q = +\infty.
\end{cases}
$$

\begin{theorem}
Let $\alpha > \beta$ be two numbers such that $g(t, \alpha) < 0$ and $g(t, \beta) > 0$ for all $t$ and for some $1 \leqslant k \leqslant +\infty$ there exists $f \in \Omega_{k,c}$ and for all $x \in [\beta, \alpha]$
$$
\frac{\partial g(t,x)}{\partial x} \leqslant f(t) \mbox{ \normalfont{(a.e.).} }
$$
Then system \eqref{duff} has at least one asymptotically stable solution $x(t)$ and $x(t) \in (\beta, \alpha)$ for all $t$.
\end{theorem}

As a corollary from this theorem we obtain the following result on the existence of an asymptotically stable $2\pi$-periodic solution of \eqref{eq3}. First, let us introduce the following notation
$$
\Phi(\varphi) = -\sin \varphi - \frac{a^2}{4} \sin 2\varphi.
$$
If $a^2 > 2$, then function $\Phi$ has two local maxima inside the interval $(\pi/2, 3\pi/2)$ and $\Phi > 0$ at these points; $\Phi$ also has two local minima where $f < 0$.
By $\lambda_1$ and $\lambda_2$ we denote the following numbers
$$
\lambda_1 = \frac{-1 + \sqrt{1 + 2 a^4}}{2a^2}, \quad \lambda_2 = \frac{-1 - \sqrt{1 + 2a^4}}{2a^2}.
$$
For $a^2\in (0, 2)$ two critical points (inside $[0, 2\pi]$) of $\Phi(\varphi)$ are as follows
$$
\varphi_{min}^1 = \mathrm{arccos}(\lambda_1), \quad \varphi_{max}^1 = 2\pi - \mathrm{arccos}(\lambda_1).
$$
As $a^2$ tends to $0$, value $\varphi_{min}^1$ tends to $\pi/2$ and $\varphi_{max}^1$ tends to $3\pi/2$. As $a^2$ tends to $\infty$, $\varphi_{min}^1$ tends to $\pi/4$ and $\varphi_{max}^1$ tends to $7\pi/4$. If $a^2 > 2$, then we have two more additional critical points
$$
\varphi_{max}^2 = \mathrm{arccos}(\lambda_2), \quad \varphi_{min}^2 = 2\pi - \mathrm{arccos}(\lambda_2).
$$
As $a^2$ tends to $2$, $\varphi_{min}^2$ tends to $\pi$ and $\varphi_{max}^2$ also tends to $\pi$. As $a^2$ tends to $\infty$, $\varphi_{min}^2$ tends to $5\pi/4$ and $\varphi_{max}^2 $ tends $ 3\pi/4$. 

\begin{proposition}
For any $2\pi$-periodic $F(t)$ such that $|F(t)| \leqslant 2/\pi$ and $-F(t) \cos \varphi^2_{max} < \Phi(\varphi^2_{max})$, $\Phi(\varphi^2_{min}) < -F(t) \cos \varphi^2_{min}$ for all $t$, there exists a $2\pi$-periodic non-falling asymptotically stable solution of \eqref{eq3}.
\end{proposition}

\begin{proof}
Let us put $\alpha = \varphi^2_{min}$ and $\beta = \varphi^2_{max}$. Then inequalities $g(t, \alpha) < 0$ and $g(t, \beta) > 0$ hold. We have
$$ 
g(t, \varphi) = \sin \varphi + \frac{a^2}{4} \sin 2\varphi - F(t) \cos\varphi,
$$
and
$$ 
\frac{\partial g}{\partial \varphi} = \cos \varphi + \frac{a^2}{2} \cos 2\varphi + F(t) \sin\varphi.
$$
We also have that
$$
\cos \varphi + \frac{a^2}{2} \cos 2\varphi \leqslant 0
$$
for all $\varphi \in [\varphi^2_{max}, \varphi^2_{min}]$. Finally, we can put $f(t)= 2/\pi$, $k = +\infty$ and apply Theorem 2.
\end{proof}

\subsection{Periodic solutions of the original system}

Let us now discuss the connection between the existence of periodic solutions of the averaged \eqref{eq3} and the original \eqref{eq2} systems. The basic result which establish this connection is a classical result by N.N. Bogolyubov on the existence of an integral manifold. A more detailed explanation of the results presented below can be found in \cite{bogolyubov1961asymptotic,mitropolsky1973integral}.

Given a system of ordinary differential equations in the standard form of the method of averaging 
\begin{equation}
    \label{eq33_1}
    \frac{d x}{d t}  = \varepsilon X(t, x),
\end{equation}
where $x \in \mathbb{R}^n$, $t \in \mathbb{R}$, $\varepsilon$ is a parameter, $X \colon \mathbb{R} \times \mathbb{R}^n \to \mathbb{R}^n$. Let function $X$ be a $2\pi$-periodic in $t$. We call system \eqref{eq33_1} the original system. Let us also consider the averaged system
\begin{equation}
    \label{eq33_2}
    \frac{d x}{d \tau} = X_0(x),
\end{equation}
where function $X_0$ is the average of function $X$ w.r.t. the periodic variable $t$ and $\tau = \varepsilon t$
$$
X_0(x) = \frac{1}{2\pi}\int\limits_0^{2\pi} X(t, x)\, dt.
$$

Let us assume that averaged system \eqref{eq33_2} has a periodic solution $\tilde x(\tau)$ of period $2\pi$ and the characteristic multipliers of this solution are strictly inside the unit circle. Then for sufficiently small non-negative $\varepsilon$ there exists an integral manifold close to solution $\tilde x$. Let us recall that we say that a manifold is integral if any solution starting at this manifold never leaves it. In coordinates, this integral manifold, which we denote by $S$, can be presented as follows
$$
x = f_\varepsilon(t,\theta),
$$
where $f_\varepsilon$ is $2\pi$-periodic in both variables. The period in $t$ is equal to the corresponding period of function $X(t,x)$, $\theta$ is an angular variable and we can assume without loss of the generality that function $f_\varepsilon(t,\theta)$ is $2\pi$-periodic in $\theta$.

In a sufficiently small neighborhood of the periodic solution our integral manifold is unique. Moreover, this integral manifold is asymptotically stable, i.e. any solution of the original system asymptotically tends to $S$, provided that the initial conditions for this solution are located sufficiently close to $S$. At the same time, in the general case, we cannot conclude that all solutions on $S$ are periodic for arbitrary $\varepsilon$.

However, for system \eqref{eq2} it can be proved that integral manifold $S$ is stratified into $2\pi$-periodic solutions, provided that $\varepsilon = 1/k$ and $k \in \mathbb{N}$ is a sufficiently large number. Note, that in our case angular variable $\theta$ corresponds to variable $s$, which defines the initial moment of time for our original system. Since the original and the averaged equations for $s$ coincides, then from the asymptotic stability of $S$ we obtain the asymptotic stability for every periodic solution on $S$ (provided that we do not change the initial condition for variable $s$).

In other words, if the periods of the vertical oscillation and the horizontal force are commensurable and the averaged system has an asymptotically stable $2\pi$-periodic solution, then the original system also has an asymptotically stable $2\pi$-periodic solution, provided that $k$ is large.

\section{Numerical results}

\subsection{Stability of a given periodic motion}

The above analytical results hold only in the case when the external horizontal force is weak or some additional assumptions hold. Below we present numerical results for the cases when neither $\mu$ nor $F(t)$ in system \eqref{eq3} is small.

By a direct calculation, we obtain the following result showing that any given periodic non-falling solution can be obtained as a solution of our system for some periodic horizontal time-dependent force.

\begin{lemma}
For any $2\pi$-periodic function $\varphi(t)$ such that $\varphi(t) \in (\frac{\pi}{2}, \frac{3\pi}{2})$ for all $t$, there exists a unique $2\pi$-periodic function $F(t)$ such that $\varphi(t)$ is a solution of \eqref{eq3}. Function $F(t)$ has the following form:
$$
F(t) = \frac{ \ddot \varphi(t) + \sin \varphi(t) + \mu \dot \varphi(t) + (a^2/4) \sin 2 \varphi(t) }{\cos \varphi(t)}.
$$
\end{lemma}

\begin{figure}[h!]
\begin{subfigure}{.5\textwidth}
  \centering
  \includegraphics[width=1.0\linewidth]{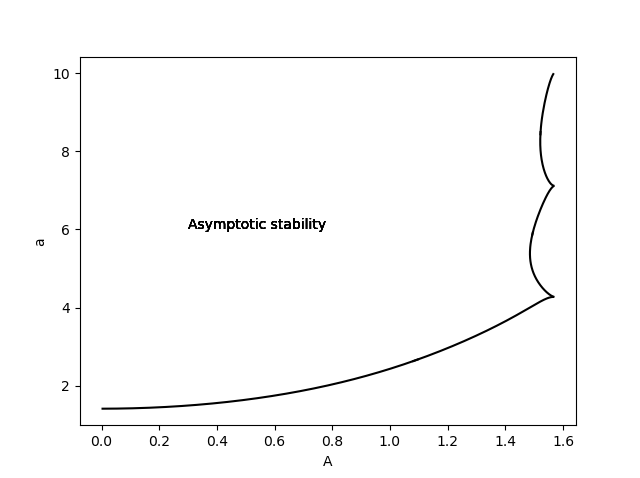}
  \caption{$\mu = 5$}
  \label{fig:sfig2}
\end{subfigure}%
\begin{subfigure}{.5\textwidth}
  \centering
  \includegraphics[width=1.0\linewidth]{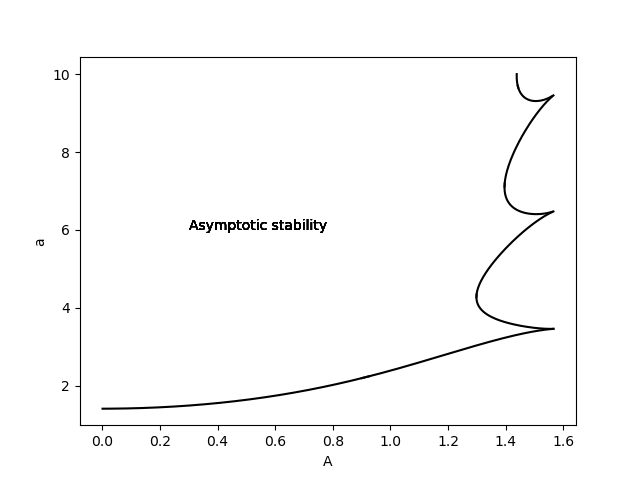}
  \caption{$\mu = 3$}
  \label{fig:sfig1}
\end{subfigure}\\
\begin{subfigure}{.5\textwidth}
  \centering
  \includegraphics[width=1.0\linewidth]{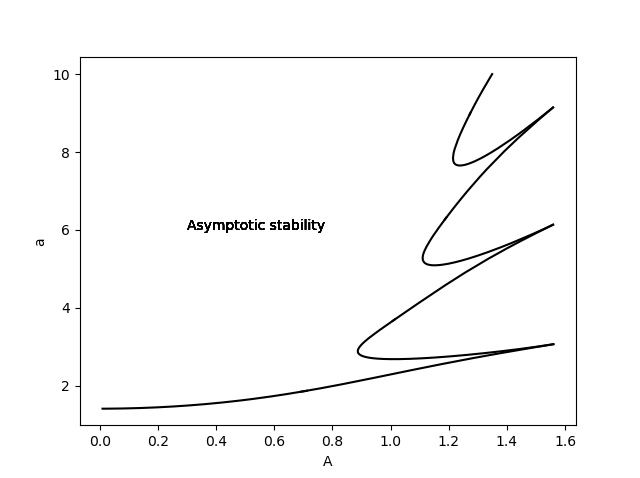}
  \caption{$\mu = 1$}
  \label{fig:sfig2}
\end{subfigure}%
\begin{subfigure}{.5\textwidth}
  \centering
  \includegraphics[width=1.0\linewidth]{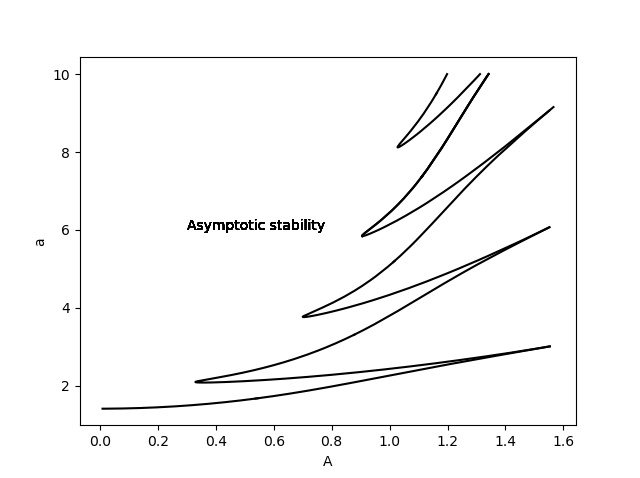}
  \caption{$\mu = 0.1$}
  \label{fig:sfig1}
\end{subfigure}
\caption{Regions of asymptotic stability for various $\mu$.}
\end{figure}

Let us assume that the values of parameters $\mu$ and $a$ in system \eqref{eq3} are given. Given any non-falling $2\pi$-periodic trajectory of motion $\varphi(t)$ ($\varphi(t) \in (\pi/2,3\pi/2)$ for all $t$), we find $F(t)$ such that $\varphi(t)$ is a solution of \eqref{eq3}. Here we put $p(t) = \dot \varphi(t)$. For instance, if we put $\varphi(t) = \pi - \pi/4 \cdot \cos t$, then we consider a periodic solution of amplitude $\pi/4$. For some values of $a$ our solution may become asymptotically stable. In this case, in the original system we also have a non-falling $2\pi$-periodic solution and these two solutions of systems \eqref{eq2} and \eqref{eq3} are close provided that $\varepsilon = 1/k$ is small. 

We will consider the following one-parameter family of solutions:
$$
\varphi_1(t) = \pi - A \cos t, \quad A \in [0, \pi/2).
$$
At the same time, we will consider various values of parameter $\mu > 0$. In Figure 2 one can find the regions of asymptotic stability for this problem.

Note that some points of the region of asymptotic stability are arbitrarily close to vertical line $A = \pi/2$. It means that for any amplitude $A < \pi/2$ we can stabilize the corresponding solution by choosing some $a$ from a bounded interval.

Moreover, the distance between the consecutive points in which the region of asymptotic stability is arbitrarily close to line $A = \pi/2$ equals $2\pi$. Based on the numerical results, it is also possible to conclude that for the first point in this sequence we have $a \to \pi$ as $\mu \to 0$.  

The most interesting effect here is that, for a given value of $A$, i.e. for a given horizontal force, the corresponding unique periodic solution can be stable for some values of $a$ and can be unstable for larger $a$. To be more precise, for $A$ close to $\pi/2$, we have multiple switchings between stable and unstable regimes as $a$ grows to infinity.

\subsection{Stability of a periodic non-falling solution for a given horizontal force}

The second problem that we study numerically can be formulated as follows.  For any $\mu > 0$ и $F \equiv 0$ there exists the vertical equilibrium of system \eqref{eq3} for which $\varphi = \pi$. This equilibrium is asymptotically stable when $a > \sqrt{2}$. It can be shown that for any $T$-periodic horizontal force $F$ there exists at least one non-falling $T$-periodic solution. If $F \equiv 0$ then the vertical equilibrium can be considered as the required periodic non-falling solution solution. We will numerically find the values of $a$ such that the corresponding periodic solution is asymptotically stable: we will start from know condition $a > \sqrt{2}$ (for $A = 0$) and will numerically continue this condition for larger $A$.

To be more precise, if our force $F$ depends on some parameter (denoted by $A$), then we will obtain a plot of $a$ as a function of $A$. For any given $A$, the corresponding value of $a$ can be described as follows: if the amplitude of the vertical oscillations is slightly less than $a$, then the corresponding periodic solution is non-stable, if the amplitude of the vertical oscillations is slightly greater than $a$, then the corresponding periodic solution is asymptotically stable. Everywhere below we assume that $T = 2\pi$.

Let us present a more detailed explanation of the algorithm. For any given force $F$ we have to find a periodic non-falling solution of system \eqref{eq3}. Let us consider the following function $\Phi \colon \mathbb{R} \times \mathbb{R} \to \mathbb{R}$
$$
\Phi(\varphi_0, p_0) = \left( (\varphi(2\pi; \varphi_0, p_0) - \varphi_0)^2 + (p(2\pi; \varphi_0, p_0) - p_0)^2 \right)^{1/2}.
$$
Here $\varphi(t; \varphi_0, p_0)$ are $p(t; \varphi_0, p_0)$ the components of the solution of system \eqref{eq3} when $\varphi(0) = \varphi_0$ и $p(0) = p_0$. The periodic solutions correspond to the zeros of $\Phi$. We will use the method of numerical continuation and will try to find zeros close to the zeros obtained for close values of $A$. In other words, we restrict the search for periodic solutions to a neighborhood of the periodic solution obtained on the previous step (for the previous value of $A$). 

When the zeros are found, we obtain numerically the monodromy matrices of the corresponding periodic solutions and check whether or not all characteristic multipliers belong to the unit circle. The required values of $a$ is determined by means of the bisection method.

Two level set plots of two different functions $\Phi$ and their zeros are shown in Figure 3.

\begin{figure}[h!]
\begin{subfigure}{.5\textwidth}
  \centering
  \includegraphics[width=1.0\linewidth]{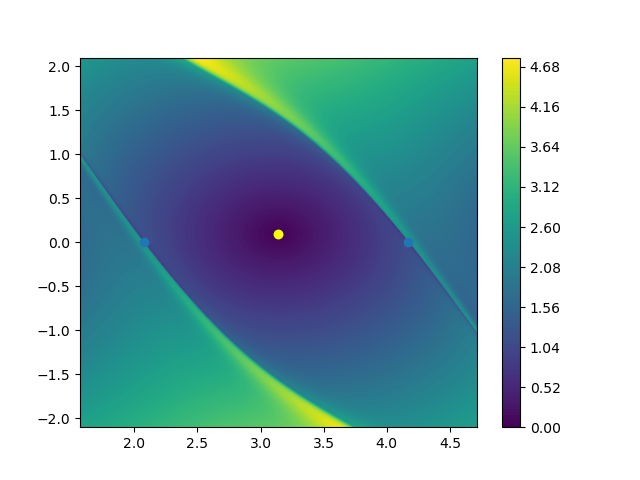}
  \caption{$A$ is small, $a$ is large ($\mu=1$, $A=0.1$, $a=2$).}
  \label{fig:sfig2}
\end{subfigure}%
\begin{subfigure}{.5\textwidth}
  \centering
  \includegraphics[width=1.0\linewidth]{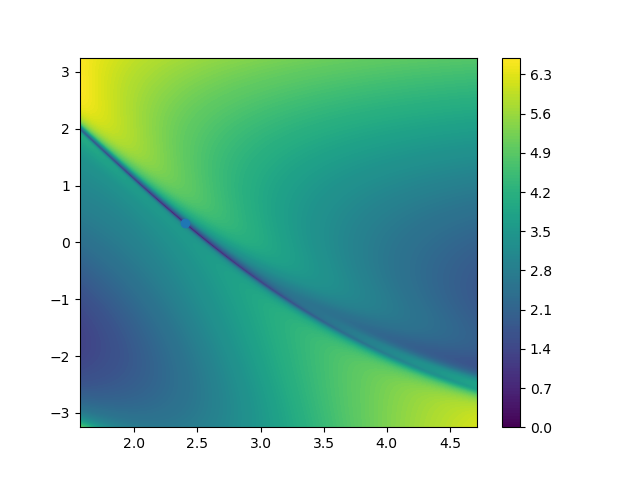}
  \caption{$a$ is small, $A$ is large ($\mu=1$, $A=2$, $a=1$).}
  \label{fig:sfig1}
\end{subfigure}
\caption{Asymptotically stable (yellow) and unstable (blue) $2\pi$-periodic solutions and the corresponding zeros of function $\Phi$.}
\end{figure}

The resulting plots are shown in Figure 4. In Figure 5 one can find how an unstable periodic solution bifurcates and two unstable and one asymptotically stable periodic solutions appear as we cross the curve presented in Figure 4 for $\mu=0.1$.

The problem of finding of all periodic non-falling solutions for a given $A$ is more difficult from the computational point of view. Therefore, we mainly focus on the numerical continuation of a known solution. The following lemma allows one to simplify the search of a periodic solution for system $\eqref{eq3}$. This Lemma has been used to obtain Figure 5.

\begin{figure}[h!]
\centering
\includegraphics[width=0.7\linewidth]{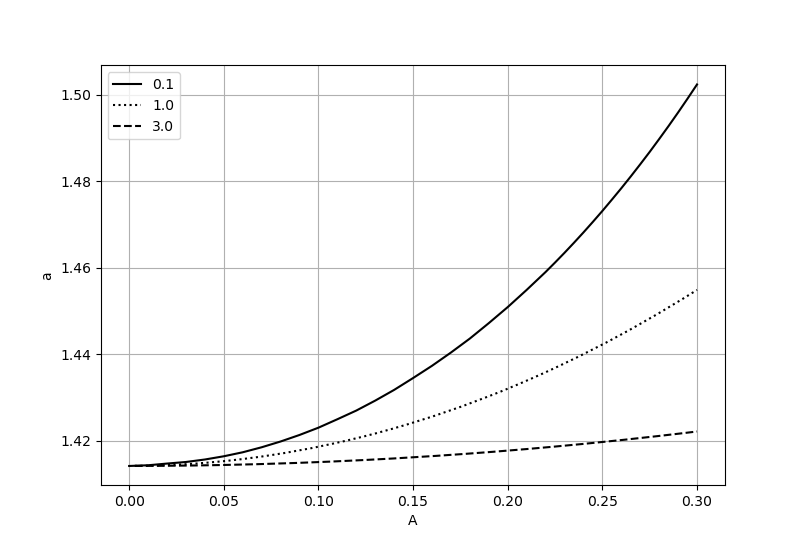}
\label{fig:sfig2}
\caption{$a$ vs. $A$ plots for $\mu$ equals $0.1$, $1$, and $3$. The regions above the plots corresponds to the values of $a$ and $A$ such that there exists at least one $2\pi$-periodic asymptotically stable non-falling solution.}
\end{figure}

\begin{lemma}
A periodic non-falling solution of system \eqref{eq3} can exist only inside the rectangle $\varphi \in [\pi/2, 3\pi/2]$, $p \in [-P, P]$, where
$$
P = \frac{1+ a^2/4 + \max|F|}{\mu}.
$$
\end{lemma} 
\begin{proof}

The first condition is the definition of a non-falling solution. Let us show that a periodic solution cannot exist in the region where $p > P$. Indeed, in all points of this region we have $\dot p < 0$ (it follows from system \eqref{eq3}). Therefore, we obtain that the solution cannot return to its initial condition. Similarly one can consider the case when $p < -P$.
\end{proof}

\begin{figure}[h!]
\begin{subfigure}{.5\textwidth}
  \centering
  \includegraphics[width=1.0\linewidth]{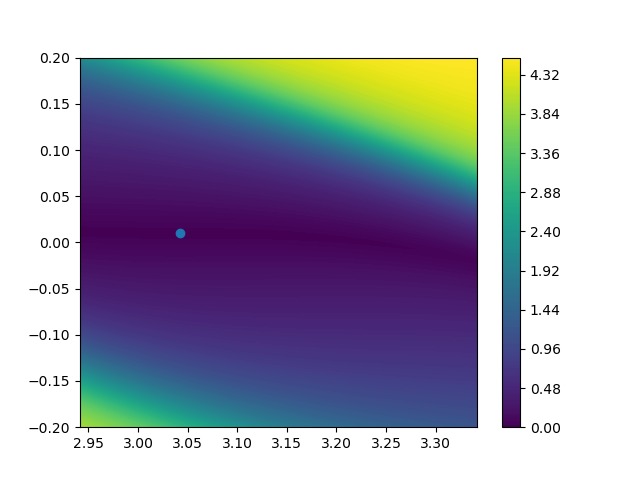}
  \caption{$a = 1.4220$}
  \label{fig:sfig2}
\end{subfigure}%
\begin{subfigure}{.5\textwidth}
  \centering
  \includegraphics[width=1.0\linewidth]{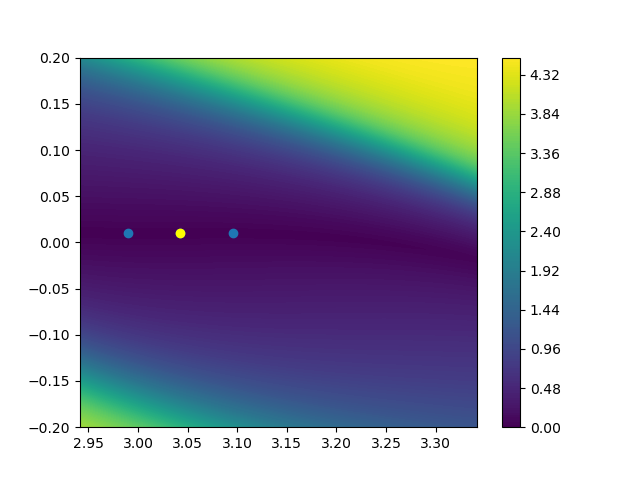}
  \caption{$a = 1.4240$}
  \label{fig:sfig1}
\end{subfigure}
\caption{Bifurcation of periodic solutions for $A = 0.1$ and $\mu = 0.1$.}
\end{figure}

\section{Conclusion}

A generalization of the well known Kapitza pendulum has been considered. In the case when there is no additional horizontal force the considered and system is obviously equivalent to the Kapitza pendulum. Some results on the existence of asymptotically stable periodic non-falling solutions are presented. In particular, the obtained numerical results (Figure 4) generalize classical conditions for the asymptotical stability of the upward vertical equilibrium of the Kapitza pendulum. 

However, the considered system is far from being completely studied. In the conclusion section we would like to outline some open problem, which will be studied elsewhere.

First, it would be interesting to prove that for any periodic (and, therefore, bounded) force $F(t)$ there exists a sufficiently large $a$ such that system \eqref{eq3} has a periodic non-falling solution provided that $\varepsilon = 1/k$ is small. It is interesting to find the correspondence between the norm of $F(t)$ and the minimal possible value of $a$ such that this solution exists.

Second, it may be of some practical interest to prove results concerning the qualitative view of plots shown in Figure 4. Is it true that these plots are always monotonously increasing? Is it possible to find several separated regions of stability for a given $A$ (similarly to the cases shown in Figure 2)? It will be interesting to compare Figure 2 with the well known Ince-Strutt diagram for the Mathieu equation (see, for instance, \cite{butikov2018analytical}). 

Third, it is known that under some additional assumptions, there exist at least two periodic solutions of the Kapitza-Whitney pendulum \cite{polekhin2020method}. Is it true that the existence of two unstable periodic non-falling solutions guarantees the existence of at least one stable periodic non-falling solution? Note that the mentioned two solutions are unstable if there is no friction in the system \cite{bolotin2015calculus}. From the above results it follows that this statement holds if $F(t)$ is small.

As a last open problem in our short list we would like to mention the problem of existence of so-called subharmonic solutions (the solution is called subharmonic if this solution is $nT$-periodic and not $kT$-periodic for all $k < n$, $k, n \in \mathbb{N}$). It would be interesting to compare the minimal values of $a$ which stabilize the solution for $T$-periodic and $nT$-periodic solutions. 

Finally, we can add that many questions that can be asked about the Kapitza pendulum can be easily reformulated for the Kapitza-Whitney pendulum.

\section*{Acknowledgements}

This work was supported by the Russian Science Foundation (Project no. 19-71-30012).

\printbibliography

@book{mitropolsky1973integral,
      title     = {Integral Manifolds in Nonlinear Mechanics (in Russian)},
      author    = {Mitropolsky, Y A and Lykova, O V},
      year      = {1973},
      publisher = {Nauka},
      address   = {Moscow}
    }

@book{malkin1949methods,
      title     = {Methods of Lyapunov and Poincar{\'e} in the Theory of Nonlinear Vibrations (in Russian)},
      author    = {Malkin, I G},
      year      = {1949},
      publisher = {Gostekhizdat},
      address   = {Moscow--Leningrad}
    }

@article{njoku2003stability,
  title={Stability properties of periodic solutions of a Duffing equation in the presence of lower and upper solutions},
  author={Njoku, Franic Ikechukwu and Omari, Pierpaolo},
  journal={Applied mathematics and computation},
  volume={135},
  number={2-3},
  pages={471--490},
  year={2003},
  publisher={Elsevier}
}

@article{torres2004existence,
  title={Existence and stability of periodic solutions of a Duffing equation by using a new maximum principle},
  author={Torres, Pedro J},
  journal={Mediterranean Journal of Mathematics},
  volume={1},
  number={4},
  pages={479--486},
  year={2004},
  publisher={Springer}
}

@article{bolotin2015calculus,
  title={Calculus of variations in the large, existence of trajectories in a domain with boundary, and {W}hitney's inverted pendulum problem},
  author={Bolotin, Sergey Vladimirovich and Kozlov, Valery Vasil'evich},
  journal={Izvestiya: Mathematics},
  volume={79},
  number={5},
  pages={894},
  year={2015},
  publisher={IOP Publishing}
}

@article{butikov2018analytical,
  title={Analytical expressions for stability regions in the {I}nce--{S}trutt diagram of Mathieu equation},
  author={Butikov, Eugene I},
  journal={American Journal of Physics},
  volume={86},
  number={4},
  pages={257--267},
  year={2018},
  publisher={American Association of Physics Teachers}
}

@book{bogolyubov1961asymptotic,
  title={Asymptotic Methods in the Theory of Non-linear Oscillations},
  author={Bogolyubov, Nikolai N and Mitropolskij, Yurij A},
  year={1963},
  publisher={Nauka},
  address={Moscow}
}

@article{polekhin2020method,
  title={The Method of Averaging for the {K}apitza--{W}hitney Pendulum},
  author={Polekhin, Ivan Yu},
  journal={Regular and Chaotic Dynamics},
  volume={25},
  number={4},
  pages={401--410},
  year={2020},
  publisher={Springer}
}

@article{stephenson1908xx,
  title={On induced stability},
  author={Stephenson, Andrew},
  journal={The London, Edinburgh, and Dublin Philosophical Magazine and Journal of Science},
  volume={15},
  number={86},
  pages={233--236},
  year={1908},
  publisher={Taylor \& Francis}
}

@article{kapitsa1951pendulum,
  title={The pendulum with an oscillating pivot point},
  author={Kapitsa, P. L.},
  journal={Uspekhi fizicheskikh nauk},
  volume={44},
  number={7},
  pages={7--20},
  year={1951}
}

@article{bogolyubov1950perturbation,
  title={Perturbation theory in nonlinear mechanics},
  author={Bogolyubov, N. N.},
  journal={Collection of Papers of Inst. Construction Mekhaniki Akad. Nauk UkrSSR},
  volume={14},
  pages={9--34},
  year={1950}
}

@article{kapitsa1951dynamic,
  title={Dynamic stability of the pendulum when the point of suspension is oscillating},
  author={Kapitsa, P. L.},
  journal={Sov. Phys. JETP},
  volume={21},
  pages={588},
  year={1951}
}

@article{artstein2021pendulum,
  title={The pendulum under vibrations revisited},
  author={Artstein, Zvi},
  journal={Nonlinearity},
  volume={34},
  number={1},
  pages={394},
  year={2021},
  publisher={IOP Publishing}
}

@article{araujo2021parametric,
  title={Parametric Stability of a Charged Pendulum with an Oscillating Suspension Point},
  author={Araujo, Gerson Cruz and Cabral, Hildeberto E},
  journal={Regular and Chaotic Dynamics},
  volume={26},
  number={1},
  pages={39--60},
  year={2021},
  publisher={Springer}
}

@article{cabral2021parametric,
  title={Parametric stability of a charged pendulum with oscillating suspension point},
  author={Cabral, Hildeberto E and Carvalho, Adecarlos C},
  journal={Journal of Differential Equations},
  volume={284},
  pages={23--38},
  year={2021},
  publisher={Elsevier}
}

@article{belyaev2021classical,
  title={Classical {K}apitsa’s problem of stability of an inverted pendulum and some generalizations},
  author={Belyaev, A. K. and Morozov, N. F. and Tovstik, P. E. and Tovstik, T. M. and Tovstik, T. P.},
  journal={Acta Mechanica},
  pages={1--17},
  year={2021},
  publisher={Springer}
}

@article{butikov2001dynamic,
  title={On the dynamic stabilization of an inverted pendulum},
  author={Butikov, Eugene I},
  journal={American Journal of Physics},
  volume={69},
  number={7},
  pages={755--768},
  year={2001},
  publisher={American Association of Physics Teachers}
}

@article{samoilenko1994nn,
  title={{N}.{N}. {B}ogolyubov and non-linear mechanics},
  author={Samoilenko, Anatoly Mykhailovych},
  journal={Russian Mathematical Surveys},
  volume={49},
  number={5},
  pages={109},
  year={1994},
  publisher={IOP Publishing}
}

@article{bardin1995stability,
  title={The stability of the equilibrium of a pendulum for vertical oscillations of the point of suspension},
  author={Bardin, BS and Markeyev, AP},
  journal={Journal of Applied Mathematics and Mechanics},
  volume={59},
  number={6},
  pages={879--886},
  year={1995},
  publisher={Elsevier}
}

@book{burd2007method,
  title={Method of averaging for differential equations on an infinite interval: theory and applications},
  author={Burd, Vladimir},
  year={2007},
  publisher={Chapman and Hall/CRC}
}
\end{document}